\documentclass[a4paper, 11pt, underruledhead]{paper}
\usepackage{amssymb}
\usepackage{amsmath}
\usepackage{amsthm}
\usepackage{math}
\usepackage[mathscr]{euscript}
\usepackage{color}
\usepackage{enumerate}
\usepackage{graphics}

\title{The Cremona-Richmond Configuration revisited and generalized}
\author{M. Pra\.zmowska, K. Pra\.zmowski}

\def\LineOn(#1,#2){\overline{{#1},{#2}\rule{0em}{1,5ex}}}
\def\collin{\sim}

\def\konftyp(#1,#2,#3,#4){\left( {#1}_{#2}\, {#3}_{#4} \right)}

\def\pointnumbsymb{\mbox{\boldmath$\nu$}}
\def\pointnumbsymbmaly{\mbox{\scriptsize\boldmath$\nu$}}
\def\pointnumb{\mathchoice{\pointnumbsymb}{\pointnumbsymb}{\pointnumbsymbmaly}{\pointnumbsymbmaly}}
\def\linenumbsymb{\mbox{\boldmath$b$}}
\def\linenumbsymbmaly{\mbox{\scriptsize\boldmath$b$}}
\def\linenumb{\mathchoice{\linenumbsymb}{\linenumbsymb}{\linenumbsymbmaly}{\linenumbsymbmaly}}
\def\ranksymb{\mbox{\boldmath$r$}}
\def\ranksymbmaly{\mbox{\scriptsize\boldmath$r$}}

\let\pointrank\rank

\def\GrasSpace(#1,#2){{\bf G}_{{#2}}({#1})}
\def\GrrSpace(#1,#2){{\bf H}_{{#2}}({#1})}
\def\GrrSpacex(#1,#2){{\bf H}^\ast_{{#2}}({#1})}
\def\GrrSpaced(#1,#2){{\bf H}^{\mbox{\footnotesize\boldmath$\circ$}}_{{#2}}({#1})}
\def\CRSpace(#1,#2,#3){{\bf R}_{{#2}}^{{#3}}({#1})}

\def\fixfield{\goth F}

\newenvironment{ctext}{%
  \par
  \smallskip
  \centering
}{%
 \par
 \smallskip
 \csname @endpetrue\endcsname
}

\newcounter{sentencex}
\def\thesentencex{\Roman{sentencex}}
\def\labelsentencex{\upshape(\thesentencex)}

\def\labelsentencey{\bf}  

\def\myend{{}\hfill{\small$\bigcirc$}}

\begin{document}

\maketitle

\def\koral#1#2#3{\binom{#1}{{#2}\ast{#3}}}
\def\koralf#1#2{\binom{{#1}{#2}}{{#1}\ast{#2}}}

\begin{abstract}
  We propose a generalization of the classical point-line Cremona-Richmond configuration
  to a configuration of points and more dimensional subspaces of a projective space,
  and present them as geometric realizations of some interesting abstract incidence
  structures.
\end{abstract}

\begin{flushleft}\small
\end{flushleft}


\section*{Introduction}

The Cremona-Richmond Configuration (shortly: C-R-Configuration) was invented
as a geometrical $\konftyp(15,3,15,3)$-configuration of points and lines
in a projective 3-space (see \cite{coxet}, \cite{vorau}).
On the other hand, it is also the famous generalized quadrangle of order $(2,2)$
(see \cite[Page 122]{payne} and \cite{obrazki}).
It can be also presented as a (combinatorial) scheme of the 2-subsets 
of a 6-set (cf. \cite{kulki}). 
Another nice example can be found in elementary geometry: C-R-Configuration
is obtained as a figure consisting of some special points in a triangle
(see \cite{evans}).
Each one of these presentations can be generalized,
and these generalizations need not necessarily be consistent/contigent.
\par
In this note we introduce a generalization of the {\em combinatorial} C-R-configuration
and discuss in a point its automorphisms. 
Then, we define a (quite natural) generalization of the {\em geometrical} C-R-configuration
and we show that the obtained structures are realizations 
of some generalized combinatorial C-R-configurations.
To some extent following that way we can realize also all the generalized C-R-configurations,
and some other related combinatorial configurations.


\section{Constructions and results: combinatorial configurations}

Let $\sub_l(X)$ denote the set of $l$-element subsets of a set $X$.
Let $k,s$ be positive integers, and
$X$ be a set with $n  = k\cdot s + m$ elements, $m \geq 0$.
For  $a \subset X$ we write $\varkappa(a) = X\setminus a$.
Let us set
\begin{multline}
  {\mathscr E}_k^s(X) \/ :=  \/ 
  {\mathscr E}  \/ := \/
  \big\{ \{ a_1\cup p,\ldots,a_s\cup p \}\colon a_1,...a_s\in\sub_k(X),\; p \in \sub_m(X),
  \\
   a_i\cap a_j = \emptyset = p \cap a_i,a_j  \text{ for all } 1\leq i<j\leq s  \big\}.
\end{multline}
%
We call the incidence structure
\begin{equation}
  \CRSpace(X,k,s) \/ := \/ \struct{\sub_{k+m}(X),{\mathscr E}_k^s(X)}
\end{equation}
%
a {\em generalized Cremona-Richmond configuration}. 
We write, shortly
\begin{ctext}
  $\CRSpace(n,k,s)$ for $\CRSpace(X,k,s)$, where $X$ is a set with $|X| = n\geq ks$. 
\end{ctext}
The number $m = |X| - ks$  is uniquely determined by the parameters $X$, $s$, and $k$.
\par
Recall also one more definition from \cite{kulki}:
a {\em generalized Sylvester system} is a structure of the form
\begin{equation}
  \GrrSpace(X,k) = \struct{\sub_{2k}(X),
  {\big\{ \{ a,b,a\setminus b \cup b \setminus a \} \colon |a\cap b| = k,\, a,b \in\sub_{2k}(X) \big\}}}.
\end{equation}
%
Let $a$ be a point of $\CRSpace(X,k,s)$. 
   The neighborhood of $a$ is the substructure of $\CRSpace(X,k,s)$
   whose points are the points on the blocks through $a$ with the point $a$ deleted,
   and whose blocks are the blocks through $a$.
\begin{fact}\label{fct:otocz}
   Assume that $|X| = ks$ and $a\in\sub_k(X)$.  
   Then the neighborhood of $a$ in $\CRSpace(X,k,s)$ is exactly
   $\CRSpace(X\setminus a,k,s-1)$.
\end{fact}

%
The following is evident
\begin{fact}\label{fct:paramy}
  Let  $|X| = n = ks+m$. 
  Then 
  $\CRSpace(X,k,s)$ is a 
  $\konftyp({(\pointnumb_k^s)},{(\pointrank_k^s)},{(\linenumb_k^s)},{s})$-configuration
  with 
  $\pointnumb_k^s = \binom{n}{k(s-1)}$,
  $\pointrank_k^s = \frac{((s-1)k)!}{(s-1)! (k!)^{s-1}} \cdot \binom{k+m}{k} =
  \frac{(n - (s-1)k)! ((s-1)k)!}{ (n- ks)! (s-1)! (k!)^s}$, 
  and
  $\linenumb_k^s = \frac{n!}{(n - ks)!  s! (k!)^s}$, 
  such that
  \begin{enumerate}[{A}1.]\itemsep-2pt
  \item
    any two blocks which have (at least) $s-1$ common points coincide, and
  \item
    two distinct blocks may have $0,1,...,s-2$ common points.
  \end{enumerate}
  %
\end{fact}
\begin{proof}
  To compute $\pointnumb^s_k$ note that, right from definition,
  $\pointnumb^s_k = \binom{n}{n-(k+m)}$.
  The formula for $\pointrank^s_k$ follows by 
  the property analogous to \ref{fct:otocz}. 
  Let $a \in\sub_{k+m}(X)$.
    Each block of $\CRSpace(X,k,s)$ through $a$ has form 
    $\big\{ a \big\} \cup \big\{  p \cup u   \colon u \in B \big\}$, 
    where $B$ is an {\em arbitrary}
  block of $\CRSpace({X\setminus a},k,s-1)$ and $p\in\sub_m(a)$.
  The value
  $\frac{((s-1)k)!}{(s-1)! (k!)^{s-1}}$ is the number 
  of the decompositions of an $n-(k+m)=k(s-1)$-element set into $s-1$ 
  undistinguishable groups, with $k$ elements in each group.
  Clearly, each block of $\CRSpace(X,k,s)$ has $s$ elements.
  With 
  $\pointnumb^s_k \cdot \pointrank^s_k = \linenumb^s_k \cdot s$
  we compute $\linenumb^s_k$.
  \par
  Let $a_1,...,a_r$ be a decomposition of a $k r$-element set $Z$ into disjoint
  $k$-sets, $k,r\geq 2$. Choose $x_i\in a_i$ for each $i =1,...,r$ and  set 
  $b_i = (a_i \setminus\{ x_i \})\cup \{ x_{i +1 \mod r} \}$.
  Then $b_1,...,b_r$ is a decomposition of $Z$ and 
  $a_{i'}\neq b_{i''}$ for all $1 \leq i',i''\leq r$.
  That way we 
  compute the number of possible common points of two blocks of $\CRSpace(X,k,s)$.
\end{proof}
The structure  $\CRSpace(3k+m,k,3)$  
coincides  with the partial linear space
$\varkappa({\GrrSpace(3k+m,k)})$,
isomorphic to the partial Steiner triple system
$\GrrSpaced(3k+m,k+m)$ introduced in \cite{kulki}.
In particular, $\CRSpace(3k,k,3) \cong \GrrSpace(3k,k) \cong \GrrSpacex(3k,k)$
(see Figure \ref{fig:h2:6}).
More particularly, $\CRSpace(6,2,3)$ is the generalized quadrangle of order $(2,2)$.

\begin{figure}[!h]
\begin{center}
  \includegraphics{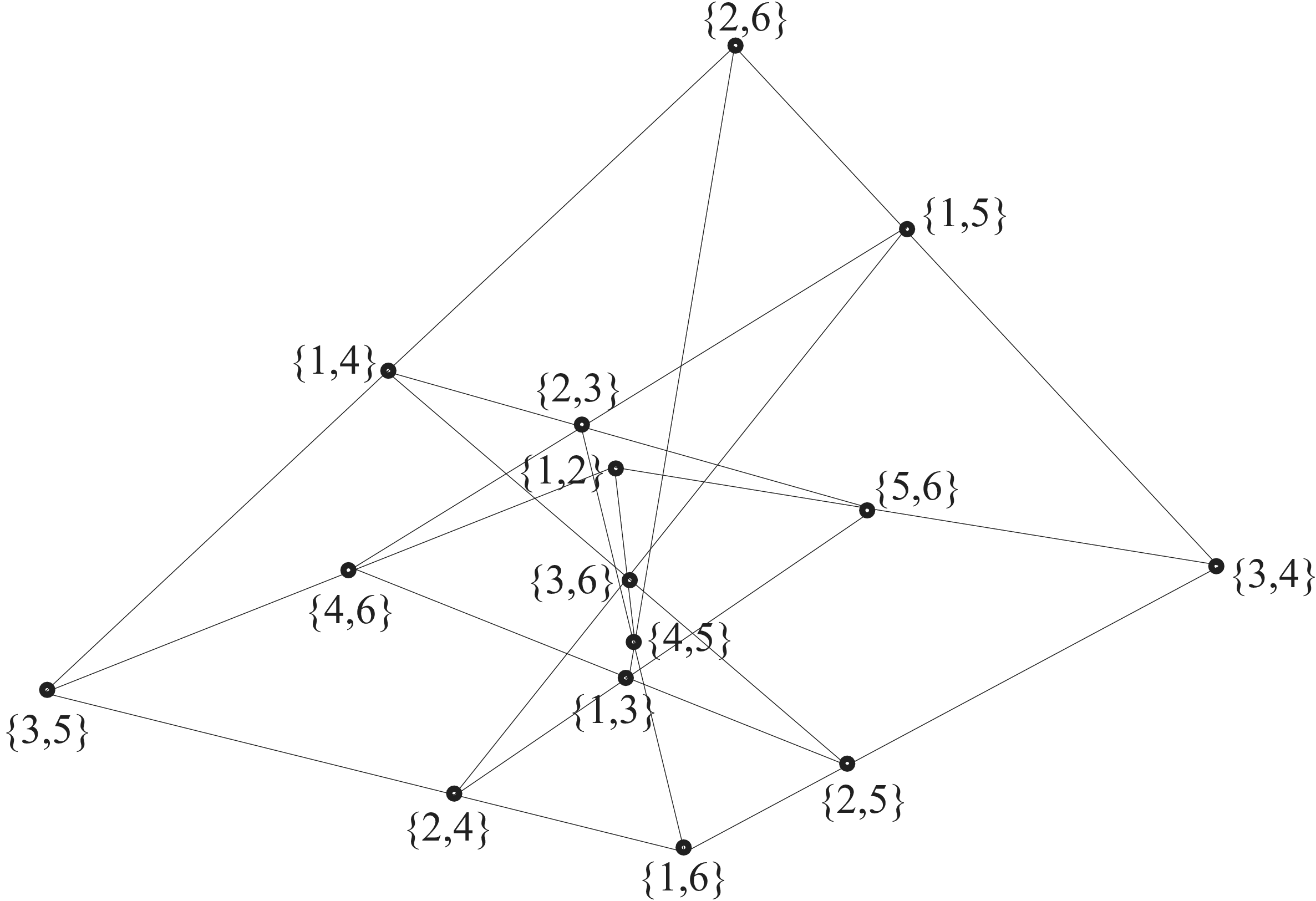}
\end{center}
\caption{A drawing of $\GrrSpacex(6,2)$ i.e. the {\em classical} Cremona-Richmond Configuration.}
\label{fig:h2:6}
\end{figure}

As a consequence of A1, A2 of \ref{fct:paramy} a structure of the form 
$\CRSpace(4k+m,k,4)$ may be considered as a weak (partial) chain space 
(cf. \cite{herzer}, \cite{belicox}).
In particular, by \ref{fct:otocz},
each weak chain structure $\CRSpace(4k,k,4)$ can be imagined as 
a block structure
obtained by nesting and closing several copies of ``generalized Sylvester Systems'' 
$\GrrSpace(3k,k)$, each one of the latter being a
famous triangle-free point-line configuration.
A few more properties of $\CRSpace(4k,k,4)$ are indicated in \ref{exm:CR-WCS}.

Let us also mention one more area of mathematics where the structures like 
$\CRSpace(n,k,s)$ appear. Namely, the binary joinablity relation on the point set
of the incidence structure $\CRSpace(ks,k,s)$ is the well-know {\em Kneser graph} $KG_{ks,k}$:
the vertices of $KG_{ks,k}$ are the elements of $\sub_k(X)$ where $|X| = ks$,
and the edges are the (unordered) pairs of {\em disjoint} elements of $\sub_k(X)$
(cf. e.g. the textbook \cite{godsil:grafy}). 

\subsection{Automorphisms}
Each $\varphi\in S_X$ determines the `image'-permutation $\widetilde{\varphi}$
of $\sub(X)$.
Let $l \leq |X|$ and  write 
\begin{ctext}
  $S_X^l = \big\{ \widetilde{\varphi}\restriction{\sub_l(X)}\colon \varphi\in S_X \big\}$.
\end{ctext}
It is evident that for each $\varphi \in S_X$ its action $\widetilde{\varphi}$ on 
$\sub_{k+m}(X)$ is an automorphism of $\CRSpace(X,k,s)$; we can say
\begin{ctext} 
   $S_X^{k+m} \subset \Aut({\CRSpace(X,k,s)})$.
\end{ctext}
The following is evident
\begin{lem}\label{lem:permut:transit}
  The permutation group $S_X^{k+m}$ acts transitively on the family of points of 
  $\CRSpace(X,k,s)$.
\end{lem}
From the previous remarks and the results of \cite{kulki} (e.g.) we have
  $\Aut({\CRSpace(3k,k,3)}) = S^k_{3k} \cong S_{3k}$.
%
%
\begin{prop}\label{prop:crem:aut}
  Let $|X| = k s$. 
  Then $\Aut({\CRSpace(X,k,s)}) = S_X^k$.
\end{prop}
\begin{proof} 
  The point is to prove $\Aut({\CRSpace(X,k,s)}) \subset S_X^k$.
%
%
Let $f\in\Aut({\CRSpace(X,k,s)})$. 
It suffices (cf. \cite{kulki}, \cite{klin}) to show that
$f$ is an automorphism of the relation $\gamma$ defined below:
\begin{ctext}
  $a \mathrel{\gamma} b \iff a,b \in \sub_k(X), \;|a\cap b|=k-1$.
\end{ctext}
On the other hand, the relation $\gamma$ can be characterized by the following formula
\eqref{def:gamma},
expressible in the language of the incidence structure $\CRSpace(X,k,s)$.
For each set $\{a_1,...,a_j\}\subset\sub_k(X)$, $j\leq s$ we write
\begin{ctext}
  $\varpi(a_1,...,a_j)$ iff there exists $C\in{\mathscr E}_k^s(X)$ with $a_1,...a_j\in C$.
\end{ctext}
Let $a,b \in \sub_k(X)$.
\begin{multline}\label{def:gamma}
  a \mathrel{\gamma} b \iff \neg \varpi(a,b) \Land
  (\exists a_{4},...,a_{s}\in\sub_k(X))
  \Big[ \varpi(a_{4},...,a_{s},a)  \land \varpi(a_{4},...,a_{s},b)
  \\
  \land \big|\big\{ c\in\sub_k(X)\colon 
  \varpi(a_4,...,a_s,c) \land \varpi(a,c) \land \varpi(b,c)  \big\}\big|
  = \binom{2k-1}{k}\Big].
\end{multline}
  Since $f$ preserves $\varpi$ it preserves $\gamma$ as well and thus it is
  determined by a permutation of $X$.
\end{proof}
\begin{rem}
 Proposition \ref{prop:crem:aut} does not remain valid for arbitrary $|X|$.
 In \cite[Remarks 4.6 and 4.7]{kulki} it was proved that
 $S_7 \subsetneq \Aut({\GrrSpace(7,2)})$ and
 $S_8 \subsetneq \Aut({\GrrSpace(8,2)})$.
 Evidently, $\Aut(\varkappa({\GrrSpace(X,l)}))\cong\Aut({\GrrSpace(X,l)})$,
 so $S_7 \subsetneq \Aut(\CRSpace(7,2,3))$ and
 $S_8 \subsetneq \Aut(\CRSpace(8,2,3))$. The problem to determine
 $\Aut(\CRSpace(ks + m,k,s))$ with $m > 0$ needs more specific methods.
\end{rem}

\section{Geometric realizations}

Let us consider a projective space 
${\goth P} = PG(n-2,{\fixfield})$ ($n\geq 5$)
with an arbitrary coordinate field $\fixfield$, 
let $F$ be the corresponding set of scalars,
and let 
$Q = \{ q_0,\ldots,q_{n-1}  \}$ be a projective frame in $\goth P$.
Set $X = \{ 0,\ldots,n-1 \}$. 
With each subset $u\in\sub(X)$ we associate the projective subspace 
$Q_u = \overline{ \{ q_i\colon i \in u \}}$ of $\goth P$; 
clearly,
$\dim(Q_u) = |u|-1$ for a proper subset $u$.
The reader can find immediately connections with base subsets of \cite{bassub} 
(and \cite{pankov}): in case of \cite{bassub} a projective base of ${\goth P}$
is used in place of a projective frame $Q$, with remaining definitions unchanged.

Let $U_1,U_2$ be subspaces of a vector space.
The following formula is well known:
\begin{equation}\label{wz:dim}
  \dim(U_1+U_2) = \dim(U_1) + \dim(U_2) - \dim(U_1\cap U_2).
\end{equation}
From this we easily derive that the same formula holds for subspaces of a projective space.

\def\projsubset{\cal P}
\def\projsub(#1,#2){{\projsubset}_{#1}({#2})}
Consequently, 
\begin{equation}\label{war:def:reprez}
  \text{\em for each nonvoid proper subset }a\ \text{\em of } X\ \text{\em the set }
  Q_a \cap Q_{\varkappa(a)} \ \text{\em is a single point }p_a. 
\end{equation}
In the sequel we shall frequently use the map $p\colon a \longmapsto p_a$
defined in \eqref{war:def:reprez},
  $$p\colon \sub(X)\setminus\{ \emptyset,X \} \longrightarrow {\goth P}.$$ 
Let us write 
  $\projsubset$ for the set $p(\sub(X)\setminus\{ \emptyset,X \})$. 
More precisely, one should use the notation
$\projsub(Q,{\goth P})$ to indicate all the parameters involved. 
Since the space $\goth P$ is homogeneous, 
there is a collineation of $\goth P$ that maps 
$\projsub(Q_1,{\goth P})$ onto $\projsub(Q_2,{\goth P})$ 
for each two frames $Q_1,Q_2$ in $\goth P$ 
and thus the parameter $Q$ is inessential in the definition of $\cal P$. 
The parameter $\goth P$ may be, sometimes, essential.

Note that $p_{\{ i \}} = q_i$ for each $i \in X$.
Directly from the definition we get $p_a = p_{\varkappa(a)}$.
%
\par
The projective spaces considered in this note are at least threedimensional and thus
each one can be represented as the projective space over a vector space.
Let $\field V$ be a vector space with a basis $e_1,\ldots,e_{n-1}$, over a field $\fixfield$.
Set 
$e_0 = \sum_{i=1}^{n-1} e_i$. Then the family 
  \begin{ctext}
    $Q \quad = \quad (q_i = \gen{e_i} = F e_i\colon i=0,...,n-1)$ 
  \end{ctext}
is a frame in the projective space 
  ${\goth P} = {\bf P}({\field V}) = PG(n-2,\fixfield)$ over $\field V$.
\begin{prop}\label{prop:repr:coo}
  Let $a \in\sub(X)$, $a \neq \emptyset, X$, and $0 \notin a$.
  Then 
  \begin{equation}\label{eq:repr:coo}
    p_{a} = p_{\varkappa(a)} = \gen{\sum_{i \in a} e_i}.
  \end{equation}
\end{prop}
\begin{proof}
  Without loss of generality we can assume that $a = \{ 1,...,k \}$.
  If $k=n-1$ then $p_{\varkappa(a)}= p_{\{ 0 \}} = q_{0}$, so \eqref{eq:repr:coo} is valid.
  Assume that $k < n-1$.
  The set of solutions of the equation
  \begin{equation}\label{eq:frame:analytic}
    \alpha_1 e_1 + \cdots + \alpha_k e_k = 
    \beta_0 (\sum_{i=0}^{n-1} e_i) + 
    \beta_{k+1} e_{k+1} + \cdots + \beta_{n-1} e_{n-1}
  \end{equation}
  with the unknown all the $\alpha_i,\beta_j$ is, clearly,
  the projective point $\gen{\sum_{i=1}^{k} e_i}$
  ($\alpha_1 = ... = \alpha_k = \beta_0 = -\beta_{k+1} = ... = -\beta_{n-1}$).
  The left-hand-side of \eqref{eq:frame:analytic} represents an element of 
  $Q_a$ with $a = \{1,...,k\}$, 
  and the right-hand-side of \eqref{eq:frame:analytic} represents an element of 
  $Q_b$ with $b = \{0,k+1,...,n-1\}$, complementary to $a$. Therefore, finally,
  $\gen{\sum_{i=1}^{k} e_i} = p_{a} = p_b$.
\end{proof}

\subsection{Embedding of $\CRSpace(ks,k,s)$}
Note evident observation.
\begin{lem}\label{lem:suma:dwie}
 Let $a',a''\in\sub_k(X)$ be disjoint.
 Then 
   $Q_{\varkappa(a')} \cup Q_{\varkappa(a'')}$ 
 spans the whole space $\goth P$, i.e.
   $\dim(Q_{\varkappa(a')}\cap Q_{\varkappa(a'')}) = |X| - 2k$.
 This dimension is equal to $k (s-2)$
 when $|X| = ks$.
\end{lem} 
Consecutively applying \eqref{wz:dim} and \ref{lem:suma:dwie}
we get
\begin{lem}\label{lem:kroj:m}
  Assume $|X| = ks$.
  Let $a_1,...,a_j\in\sub_k(X)$ be pairwise disjoint.
  Then
  \begin{equation}\label{eq:kroj:m}
    \dim(Q_{\varkappa(a_1)}\cap ... \cap Q_{\varkappa(a_j)}) = k(s-j) + j -2.
  \end{equation}
\end{lem}

%
%
As an important consequence of \ref{lem:kroj:m} and \ref{prop:repr:coo} 
we obtain the following
\begin{prop}\label{prop:crememb:gen}
  Assume that $|X| = k s$.
  The map 
  $\sub_k(X) \ni a \longmapsto p_a$ embeds $\CRSpace(X,k,s)$ into 
  ${\goth P} = PG(n-2,{\goth F})$:
  each  block $B$ of $\CRSpace(X,k,s)$ (an $s$-element set of points)
  is mapped on points on a $(s-2)$-subspace of $\goth P$. Precisely: 
  the projective dimension of the subspace $\overline{p(B)}$ is $s-2$.
  Moreover, $p_a \notin \overline{p(B)}$ when $a \notin B$ and therefore
  $p$ sends distinct blocks to distinct $(s-2)$-subspaces.
\end{prop}
\begin{proof}
  Let $B = \{ a_1,...,a_s \}$ be a block of $\CRSpace(X,k,s)$. 
  Each point $p_{a_i}$ is contained, by definition \eqref{war:def:reprez}, in
  the subspace $Q_{\varkappa(a_i)}$.
  And $p_{a_i}$ is contained also in $Q_{\varkappa(a_j)}$ for $i\neq j$: 
  indeed, $a_i\subset \varkappa(a_j)$ and thus
  $p_{a_i} \in Q_{a_i}\subset Q_{\varkappa(a_j)}$.
  Consequently,   
    $p_{a_1},...,p_{a_s}\in Q_{\varkappa(a_1)}\cap ... \cap Q_{\varkappa(a_s)}$.
  From \eqref{eq:kroj:m} we get that
    $\dim(Q_{\varkappa(a_1)}\cap ... \cap Q_{\varkappa(a_s)}) = s -2$.
  \par
  Now, suppose that $\sum_{j=1}^t \alpha_j p_{a_{i_j}} = \theta$ is the zero vector of 
  $\field V$  for some scalars $\alpha_j$, $t<k$. 
  Without loss of generality we can assume that
  $0\notin a_{j_1},...,a_{j_t}$ and then from \ref{prop:repr:coo} we get $\alpha_j = 0$
  for each $j$. So, no proper subset of the points in $p(B)$ is projectively
  dependent.  
  \par
  Suppose that $p_a \in \overline{p(B)}$, $B = \{ a_1,\ldots,a_s \}$ is a block of
  $\CRSpace(X,k,s)$, $B\not\ni a$ is a point of $\CRSpace(X,k,s)$.
  Without loss of generality we can assume that $0\in a_s$.
  Assume $0\notin a$. Then \ref{prop:repr:coo} yields
  $\sum_{j\in a} e_j = \sum_{i=1}^{s-1}(\alpha_i \sum_{j\in a_i} e_j) +
  \alpha_s \sum_{j\in a_1 \cup\ldots\cup a_{s_1}} e_j$.
  For each $i=1,\ldots,s-1$ there is $j_i\in a_i\setminus a$, which gives
  $0\cdot e_{j_i} = \alpha_i\cdot e_{j_i} + \alpha_s\cdot e_{j_i}$
  and therefore $\alpha_i = -\alpha_s$.
  Finally, we conclude with $\sum_{j\in a} e_j = \theta$.
  Analogous contradiction arizes when we put $0\in a$.
  This completes the proof of our claim.
\end{proof}
In view of \ref{prop:crememb:gen} with each block $B$ of $\CRSpace(ks,k,s)$
we can uniquely associate a $(s-2)$-dimensional subspace $p_B = \overline{p(B)}$ 
of $\goth P$ such that the projective configuration (embedded into $\goth P$)
\begin{ctext}
  $\struct{ \projsub({},{\goth P}), \{ p_B\colon B\in{\mathscr E}_k^s(sk) \} }$
\end{ctext}
is isomorphic to $\CRSpace(sk,k,s)$.
\par
With computations analogous to those in the proof of \ref{prop:crememb:gen} we can 
formulate a direct characterization of the subspaces $p_B$:
\begin{prop}\label{prop:repr:blocks}
  Let $B = \{ a_1,\ldots,a_s \}$ be a block of $\CRSpace(X,k,s)$, $|X| = ks$;
  let $0\in a_s$. Then
  \begin{equation}\label{eq:repr:blocks}
    p_B = \gen{\sum_{j\in a_1} e_j,\ldots,\sum_{j\in a_{s-1}} e_j}.
  \end{equation}
\end{prop}
Note that, as an evident consequence of \ref{prop:crem:aut}
the embedding defined in \ref{prop:crememb:gen} is {\em free} in the following
sense:
{\em each automorphism of $\CRSpace(ks,k,s)$ extends to a projective collineation of
$PG(ks-2,\fixfield)$}.

\par
One more property of the above embedding may be worth to note.
\begin{prop}\label{prop:minspan}
  Let $1< k$ be an integer. Assume that the characteristic of $\fixfield$ does not divide $k$.
  Let $A\subsetneq X$ and $2k < |A|$. Then the two projective subspaces of $\goth P$:
  the set $Q_A$ spanned by the points $q_x = p_{\{ x \}}$ with $x \in A$,
  and the subspace spanned by the set $\{ p_a\colon a \in \sub_k(A) \}$,
  coincide.
\end{prop}
\begin{proof}
  The claim is a tautology for $k=1$, so let $k \geq 2$. 
  Set $v = |A|$. Since $v < |X|$, we can assume that $0\notin A$ 
  and make use of \ref{prop:repr:coo},
  and then it is clear that $\{ p_a\colon a \in \sub_k(A) \} \subset Q_A$.
  This, what remains to prove consists in the following 
  \par\noindent{\bf Lemma}. 
  {\em
  Let $B = \{e_1,...,e_v\}$ be a basis of a vector space $U$. Then the family of all the sums
  $e_{i_1}+...+e_{i_k}$ with $1\leq i_1<..i_k\leq v$ spans the space $U$.
  }
  \par\noindent
  {Proof of the lemma}:
  Let $U_0$ be the subspace spanned by the sums as above.
  First, observe that 
  $(e_{i_1}+...+{e_{i_{k-1}}}+ e_{i'}) - (e_{i_1}+...+{e_{i_{k-1}}}+ e_{i''})\in U_0$ for 
  each pairwise distinct basis vectors $e_{i_1},...,{e_{i_{k-1}}},e_{i'},e_{i''}$. 
  This proves that $e' - e'' \in U_0$
  for arbitrary $e',e''\in B$.
  To close the proof note that
  $e_i = \frac{1}{k}\big( (e_i + e_{i_2} + ... + e_{i_k}) + 
  ((e_i - e_{i_2}) + ... + (e_i - e_{i_k})\big) \in U_0$.
\end{proof}

\begin{note}\normalfont
The above embedding $p$ in the case of $\CRSpace(6,2,3)$ is the well known Cremona-Richmond 
Configuration of points and lines of a $(2\cdot 3 -2) = 4$ dimensional projective space 
(cf. \cite{coxet}); in the particular case 
${\fixfield} = Z_2$, it is precisely the embedding of $\GrrSpace(6,2)$ into $PG(4,2)$ 
considered in \cite{kulki}. 
On the other hand note (comp. \ref{prop:minspan}) that $\CRSpace(6,2,3)$ can be also
embedded into a less-dimensional space $PG(3,2)$. 
In the case when we embed $\CRSpace(6,2,3)$ into a projective space over $Z_2$ we 
have the following irregularity: Let $A = \{ 1,2,3 \}$. Then $\dim(Q_A) = 3$
and 
$\dim(\overline{ \{ p_a\colon a \in \sub_2(A) \}}) = 2$.
Firstly, this yields that after (C-R-) embedding, some more collinearities may appear
than those inherited from $\CRSpace(6,2,3)$ (see \ref{exm:fano:3}).
\end{note}

\begin{exm}
  Note, moreover, that the embedding defined in \ref{prop:crememb:gen} 
  gives us a projective point-line configuration 
  isomorphic to $\GrrSpace(3k,k)$ for arbitrary $k \geq 3$, embedded into an arbitrary 
  $(3k-2)$-dimensional projective space.
\end{exm}

\begin{exm}\label{exm:CR-WCS}
  Let us pay attention to a weak chain structure 
  ${\goth M} := \CRSpace(X,k,4)$ with $|X| = 4k$.
  As already noted, the neighborhood of a point of $\goth M$ is the configuration 
  $\GrrSpacex(3k,k)$: the ``classical'' Cremona-Richmond configuration.
  Let us quote some evident `synthetic' properties of $\goth M$. 
  Their proofs consist in straightforward computations.
  \begin{itemize}\def\labelitemi{--}\em
  \item
    Let $A_1,A_2,A_3$ be blocks of $\goth M$, which pairwise intersect in points: 
    $a_i \in A_{i},A_{i+1 \mod 3}$.
    Then $a_1,a_2,a_3$ are on a block 
    {\em(cf. the axiom 
    `{\em $a_i \collin a_j $ for $i,j\in\{ 1,2,3 \}$ implies $a_1,a_2,a_3$ on a block}'
    of the M{\"o}bius geometry)}.
  \item
    Let $B_0,B_1,B_2$ be blocks of $\goth M$. 
    If $|B_0 \cap B_i| = 2$ for $i=1,2$ then 
    $B_1 = B_2$  or $B_1 \cap B_2 \subset B_0$.
    {\em This yields that no ``chain-plane'' can be spanned in $\goth M$ 
    by a pair of blocks that have two common points. Moreover
    {\em the Miquel Axiom and the Bundle Axiom are valid in $\goth M$}, 
    but they are {\em trivially valid}:
    no system of points and blocks of $\goth M$ satisfies the assumptions 
    of the axioms in question.}
  \item
    Let $B_0 \cap B_1 = \{ a \}$ and $a \neq b \in B_1$ for points $a,b$ and blocks $B_0,B_1$ 
    of $\goth M$.
    Then there is  a block $B_2$ such that $b \in B_2 \neq B_1$, $B_0 \cap B_2 = \{ a \}$.
    {\em Therefore it is hard to imagine how to define the ``tangency'' 
    relation in $\goth M$.}
  \end{itemize}
\end{exm}

\begin{prob}
  Is it possible to represent a weak chain structure $\CRSpace(X,k,4)$ on a sphere for 
  some $X$?: represent with blocks interpreted as circles!
  In view of \ref{exm:CR-WCS} it may be a hard task.
\end{prob}

\subsection{Embedding of $\CRSpace(ks+m,k,s)$, $m>0$}
Of course, the map/representation $p$ defined on subsets of a set $X$
can be considered in the case $|X|-ks > 0$  
as well.

\begin{prop}\label{prop:brakemb}
  Let $|X| = n = k\cdot s + m$, $m >0$.
  If the characteristic $\mathrm{char}(\fixfield)$ of $\fixfield$ does not divide $s-1$ 
  then $p$ does not yield any embedding of 
  $\CRSpace(n,k,s)$ into $PG(n-2,\fixfield)$
  with the blocks interpreted as $(s-2)-subspaces$ (comp. \ref{prop:crememb:gen}).
  If $\mathrm{char}(\fixfield) \mathrel{|} (s-1)$ then $p$ defined by \eqref{war:def:reprez}
  is a representation of $\CRSpace(n,k,s)$ in $PG(n-2,\fixfield)$.
\end{prop}
\begin{proof}
  Let 
    $B = \{ a_1\cup d,\ldots,a_s\cup d \} \in   {\mathscr E}_k^s(X)$.
  So, 
    $a_1,...a_s\in\sub_k(X)$,  $d \in \sub_m(X)$, 
  and
    $a_i\cap a_j = \emptyset = d \cap a_i,a_j$ for all  $1\leq i<j\leq s$.
  Write $c_i = a_i \cup d$.
  Adopt the notation of \ref{prop:repr:coo}; without loss of generality we can 
  assume that $ 0 \in d$ 
  and then 
    $p_{c_i} = p_{\varkappa(c_i)} = \sum\{ e_j\colon j=1,...,n,\,j\notin c_i \}$.
  \par
  {\em Suppose that $p(B)$ is a dependent set of points of $PG(n-2,\fixfield)$}.
  Then there are scalars $\alpha_i\in F$ not all zero such that 
    $\sum_{i=1}^{s} \alpha_i p_{c_i} = \theta =$ 
  the zero vector.
  Let us compute:
  %
  \begin{eqnarray*}
    \sum_{i=1}^s \alpha_i p_{c_i} = & \strut & 
         \sum_{i=1}^s (\alpha_i\sum_{j\notin c_i} e_j)
  \\
   \strut= &  & \alpha_1 \sum e_j \colon j \in \not{a_1} \cup a_2 \cup \ldots \cup a_s
  \\
   \strut & + & \alpha_2 \sum e_j \colon j \in a_1 \cup \not{a_2} \cup \ldots \cup a_s
  \\
   \strut & + & \vdots
  \\
   \strut & + & \alpha_s \sum e_j \colon j \in a_1 \cup a_2 \cup \ldots \cup \not{a_s}
  \\
   \strut = & & \sum_{i=1}^s \beta_i(\sum_{j\in a_i} e_j).
  \end{eqnarray*}
  where
  $\beta_i = \beta_0 - \alpha_i$, $\beta_0 = \sum_{i=1}^s \alpha_i$.
  So, $\beta_i = 0$ for $i =1,...,s$. This yields 
  \begin{ctext}
    $\sum_{i=1}^s\beta_i = (s-1)\beta_0 = 0$.
  \end{ctext}
  Two cases are possible.
  Either $s-1 = 0$ holds in $\fixfield$ (i.e. the characteristic of $\fixfield$ divides $s-1$)
  and in this case we search for a non zero solution of the system 
  \begin{equation}\label{systm:rownania}
    \left\{ \sum_{j=1,j\neq i}^{s} \alpha_j = 0 \colon i = 1,\ldots,s \right\}
  \end{equation}
  of linear equations. The matrix of this system is the characteristic matrix of 
  the form $M = [1]_s - 1 \cdot \Delta_s$, where 
  $[1]_s$ is the $s\times s$ matrix with all the entries $1$, and 
  $\Delta_s$ is the diagonal $s\times s$ matrix. It is a folklore that
  $\det(M) = (-1)^{s+1}(s-1)$. In our case $\det(M) = 0$ in $\fixfield$ and thus the system
  \eqref{systm:rownania} has indeed a non zero solution.
  \par\noindent
  It remains to show that if $B = \{ a_1,\ldots,a_s\}$ is a block of 
  $\CRSpace(X,k,s)$ then no system 
  $p_{a_{i_1}},\ldots,p_{a_{i_t}}$ with $1\leq i_1<\ldots<i_t\leq s$, $t < s$ is 
  projectively dependent.
  To this aim one can assume, without loss of generality, that 
  $0 \notin a_{i_1},...,a_{i_t}$ and apply the reasoning analogous to 
  that in \ref{prop:crememb:gen}.
  \par
  Or $s-1\neq 0$ and thus $\beta_0 = 0$. This, in turn leads to $\alpha_i = 0$
  for $i=1,...,s$, which means that the set $p(B)$ is independent, so 
  $\dim(\overline{p(B)}) = s-1$ and $p$ is not a required embedding.
\end{proof}

Let us quote a few examples:
\begin{exm}\strut
\begin{sentences}
\item
  The map $p$ yields a projective embedding of the structure
  $\CRSpace(n,k,3)$ into $PG(n-2,2)$ for each $n\geq 3k$. Compare the embedding 
  $\mu''$ of $\GrrSpace(n,k)$ into $PG(n-2,2)$ defined in \cite{kulki}.
\item
  The map $p$ yields a projective embedding of the weak chain structure
  $\CRSpace(n,k,4)$ onto a point-plane subconfiguration of $PG(n-2,3)$
  for each $n \geq 4k$.\myend
\end{sentences}
\end{exm}
\begin{exm}
  The classical Cremona-Richmond configuration should be, perhaps, 
  considered as a configuration embedded into a projective space {\em over the reals}.
  In this case, clearly, \ref{prop:brakemb} and \ref{prop:crememb:gen}, \ref{prop:minspan}
  can be applied and we get that $p$ is a realization of the generalized 
  Cremona-Richmond configuration $\CRSpace(ks+m, k,s)$ in the real projective
  $(ks+m-2)$-space iff $m=0$.
  This does not mean, automatically, that there is no free projective realization of
  a Cremona-Richmond configuration $\CRSpace(ks+m,k,s)$ with any $m > 0$
  in a real projective space, but we
  {\em claim} that such a realization does not exist.
\myend
\end{exm}

%

\subsection{Final remarks and examples}

One can, finally, pay attention to the map $p$ defined over the whole set
$\sub(X)\setminus\{ X,\emptyset \}$. Here we give some remarks on the subject.
\begin{prob}
  Determine the projective lines (more generally: `dependencies', i.e. families
  which span in $\goth P$ a subspace with the dimension smaller than the cardinality
  of the family minus 1) 
  which join the points 
  in 
    $\{ p_a\colon a\in\sub_k(X) \} =: \projsub(k,{\goth P}) \subset \projsub({},{\goth P})$ 
  ($k$ is fixed)
  in the projective space $\goth P$ introduced at the beginning of this Section.
\end{prob}

\begin{exm}\label{exm:5-2gen}
  Let us apply the representation 
  $p \colon \sub_k(X)\longrightarrow \projsub(k,{\goth P})$
  defined in \eqref{war:def:reprez},
  where $|X| = 5$. 
  Then the constructed points are points
  of the space ${\goth P} = PG(3,{\fixfield})$ for a field $\fixfield$. 
  Generally, the only at least 3-element sections of lines of $\goth P$ with the 
  set $\projsub({},\goth P)$
  have form
\begin{eqnarray*}
  \left\{ p_a,p_b,p_c \right\} & \text{ where } &
  X = a \cup b \cup c, \; |a| = 2 = |b|,\; |c| = 1,
  \text{ and }
  \\
  \left\{ p_a,p_b,p_c \right\} & \text{ where } & c = a \cup b,\; |a| = 1 = |b|.
\end{eqnarray*}
  The combinatorial schema consisting of such blocks may seem interesting
  (cf. Figure~\ref{fig:5-frame}): 
  in particular, it contains many Desargues subconfigurations.
  Indeed, consider any of following schemas:
\begin{ctext}
  $ p_{\{i,j\}} \left[ \begin{array}{cc|c}
    p_{\{k\}} & p_{\{l,m\}} & p_{\{j,l\}}
    \\
    p_{\{i\}} & p_{\{j\}} & p_{\{m\}}
    \\
    p_{\{k,m\}} & p_{\{l\}} & p_{\{i,k\}}
  \end{array}\right]$,
  where $\{ i,j,k,l,m \} = \{ 1,2,3,4,5 \}$.
  \par
  The two first columns of the matrix consist of two triangles, perspective
  rays consist of two points in each of the rows and the point ``ahead'' the matrix.
\end{ctext}
  However, this structure is not homogeneous: the point 
  $p_a$ has rank $\left\{\begin{array}{ll}7 & \text{when }|a| = 1\\
  4 & \text{when }|a| = 2\end{array}\right.$.
  This yields that the automorphism group of it is trivially determined by the symmetric 
  group $S_X$.
\myend
\end{exm}

\begin{figure}
\begin{center}
 \includegraphics{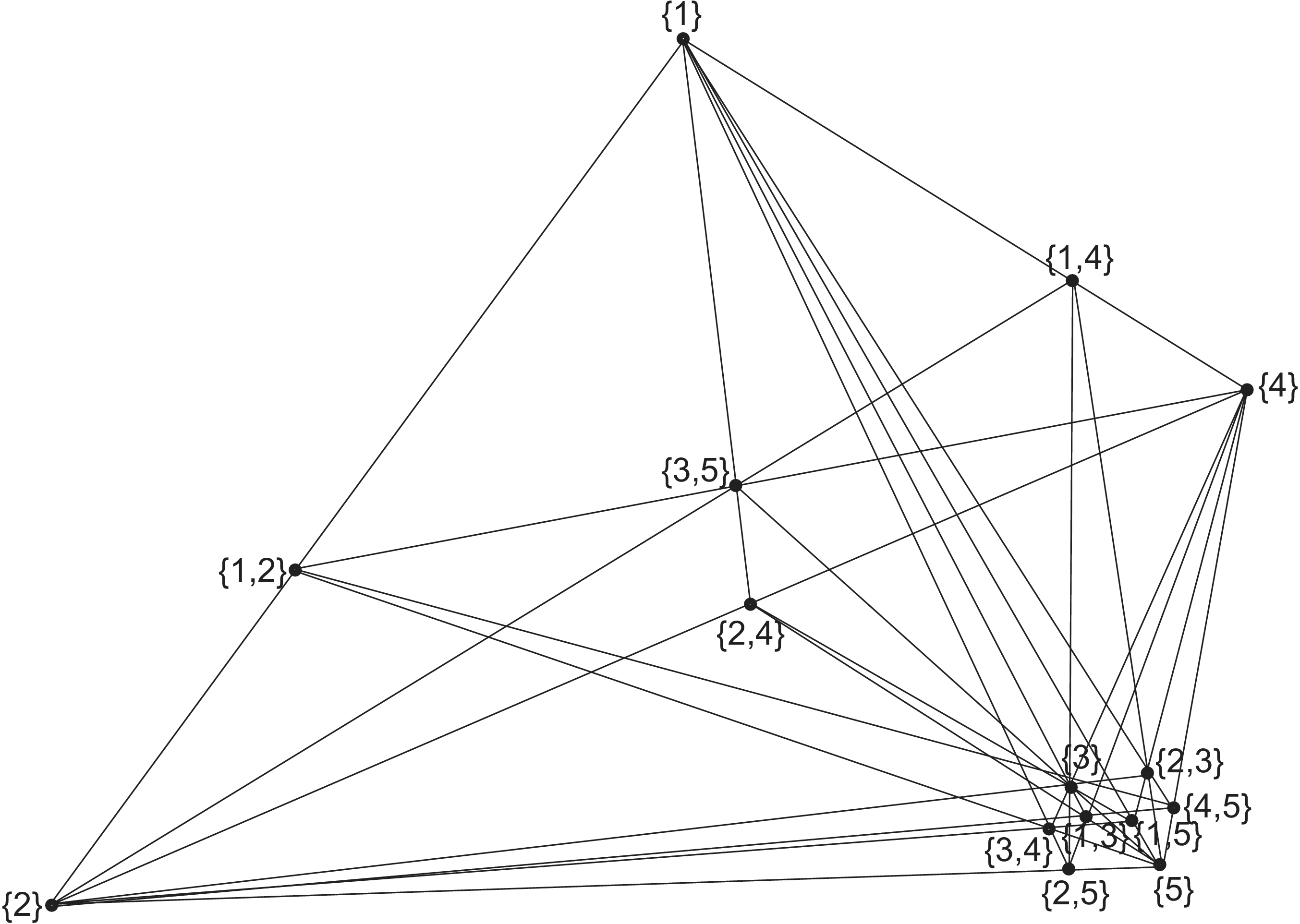}
\end{center}
\caption{Illustration to \ref{exm:5-2gen}}
\label{fig:5-frame}
\end{figure}

Let us point out also that the considered representation $\projsub({},{\goth P})$
{\em may depend} on
the coordinate field of the space $\goth P$.
\begin{exm}\label{exm:fano:3}
  Let ${\fixfield} = Z_2$. Recall that the points of $\goth P$ are the nonzero vectors
  of a vector space over $Z_2$. 
  A straightforward computation gives that
  the third point on the line through distinct points $a,b$ is the (vector-sum)
  $a+b$. So, a point $p_{\{ i,j \}}$ coincides with $q_{i}+q_{j}$ for distinct $i,j$.
  In particular, for each 3-set $Z = \{ i,j,l \}$ of indices, the projective points 
  $p_{\{ i,j \}}$,
  $p_{\{ j,l \}}$, and
  $p_{\{ l,i \}}$
  are on a line of $\goth P$. 
  Let the set $X$ of indices contain exactly 5 elements, so, consider 
  $\projsub(2,{PG(3,2)})$. Then the map $p$ embeds the Desargues configuration 
  $\GrasSpace(5,2)$ into $PG(3,2)$. The obtained Desargues subconfiguration
  completes the configuration on Figure \ref{fig:5-frame} to the projective
  space $PG(3,2)$.
  In particular, there are more lines in $PG(3,2)$ which join the points in
  $\projsub({},{PG(3,2)})$ than those listed in \ref{exm:5-2gen}.
\end{exm}
\begin{exm}\label{exm:repr:coo}
  The reasoning of \ref{exm:fano:3} can be presented in a slightly more general way.
  The analytical representation defined in \ref{prop:repr:coo} 
  allows us to show, in particular, the following
  \begin{sentences}\itemsep-2pt
  \item
    For any three digits $i,j,l$ we have 
    $p_{\{ i,j \}}$, $p_{\{ j,l \}}$, $p_{\{ i,l \}}$, and $p_{\{ i,j,l \}}$
    coplanar.
  \item
    Let ${\fixfield} = Z_3$. 
    For any four  element set $\{ i,j,l,m \}$ the points
    $p_{\{ i,j,l \}}$, $p_{\{ i,j,m \}}$, $p_{\{ i,l,m \}}$, $p_{\{ j,l,m \}}$
    in the space $PG(n,3)$, $n\geq 4$
    are coplanar.
  \myend
  \end{sentences}
\end{exm}


\bigskip

\par\noindent\small
Authors' address:\\
Ma{\l}gorzata Pra{\.z}mowska, Krzysztof Pra{\.z}mowski\\
Institute of Mathematics, University of Bia{\l}ystok\\
ul. Akademicka 2\\
15-246 Bia{\l}ystok, Poland\\
e-mail: 
{\ttfamily malgpraz@math.uwb.edu.pl},
{\ttfamily krzypraz@math.uwb.edu.pl}


\begin{thebibliography}{9}\itemsep-2pt\footnotesize
\bibitem{coxet}
  {\sc H. S. M. Coxeter},
  {\it Self-dual configurations and regular graphs},
  Bull. Amer. Math. Soc. {\bf 56}(1950), {413--455}.
\bibitem{evans}
  {\sc L. S. Evans},
  {\it Some configurations of triangle centers},
  Forum Geom. {\bf 3}(2003), {49--56}.
\bibitem{godsil:grafy}
  {\sc Ch. D. Godsil, G. Royle},
  {\sl Algebraic Graph Theory},
  Springer V. 2001.
\bibitem{vorau}
  {\sc H. Havlicek (the presenter), B. Odehnal, M. Saniga},
  {M{\"o}bius Pairs and Pauli Operators},
  [presented at:]
  Conference on Geometry -- Theory and Applications, Vorau, Austria, June 22nd, 2011.
\bibitem{herzer}
   {\sc A. Herzer}, {\it Chain geometries}, [in] {E. Buekenhout} (Ed),
   {\sl Handbook of incidence geometry, 781--842}, Elsevier, 1995.
\bibitem{klin}
  {\sc M. Ch. Klin, R. P{\"o}schel, K. Rosenbaum},
  {\it Angewandte Algebra f{\"u}r Mathematiker und Informatiker},
  VEB Deutcher Verlag der Wissenschaften, Berlin 1988.
\bibitem{kulki}
  {\sc A. Owsiejczuk, M. Pra{\.z}mowska},
  {\it Combinatorial generalizations of generalized quadrangles of order $(2,2)$},
  Des. Codes Cryptogr. {\bf 53}(2009), no. 1, {45--57}.
\bibitem{pankov}
  {\sc M. Pankov},
  {\sl Grassmannians of classical buildings},
  World Scientific, 2010.
\bibitem{bassub}
  {\sc M. Pankov},
  {\it Transformations of Grassmannians preserving the class of base sets},
  J. Geom. {\bf 79} (2004), no. 1-2, 169--176.
\bibitem{payne}
  {\sc S. E. Payne, J. A. Thas},
  {\sl Finite generalized quadrangles},
  Research Notes in Math. 110, Pitman, Boston, 1984.
\bibitem{obrazki}
  {\sc B. Polster},
  {\sl A geometrical picture book},
  Springer Verlag, New York, 1998.
\bibitem{belicox}
  {\sc M. Pra{\.z}mowska, K. Pra{\.z}mowski}
  {\it The Cox, Clifford, M{\"o}bius, Miquel, and other related 
  configurations and their generalizations},
  mimeographed.
\end{thebibliography}
\end{document}